\documentclass[a4papet,10pt]{article}
\usepackage{amsfonts,amssymb,mathtools}
\usepackage{graphicx}
\usepackage{amsmath}
\usepackage[latin1]{inputenc} 
\usepackage{bbm}
\usepackage{pstricks}
\usepackage{psfrag}
\usepackage{appendix}
\usepackage{amsthm}
\usepackage{dsfont}
\usepackage{stackrel}
\usepackage{eufrak}
\usepackage{mathrsfs}
\usepackage{enumerate}
\usepackage{cite}
\usepackage{chngcntr}
\counterwithin{figure}{section}
\usepackage[a4paper,top=30mm,bottom=30mm,left=35mm,right=35mm]{geometry}
\numberwithin{equation}{section}

\newcommand{\prob}{\mathbb{P}}
\newcommand{\Ex}{\mathbb{E}}
\newcommand{\Rl}{\mathbb{R}}

\newcommand{\cl}{\mbox{cl}\,}

\newcommand{\Is}{I_{\mathrm{s}}}

\newcommand{\Ii}{I_{\mathrm{i}}}

\newcommand{\ells}{\ell_{\mathrm{s}}}

\newcommand{\elli}{\ell_{\mathrm{i}}}
\newcommand{\rew}{\mathcal{X}}
\newcommand{\Brew}{\mathcal{B}(\mathcal{X})}
\newcommand{\BWrew}{\mathcal{B}(\Rl\times\mathcal{X})}

\newtheorem{theorem}{Theorem}[section]
\newtheorem{proposition}{Proposition}[section]
\newtheorem{lemma}{Lemma}[section]

\title{Large deviation principles for renewal-reward processes}

\author{Marco Zamparo\footnote{Dipartimento di Fisica, Universit\`a degli Studi di Bari and INFN, Sezione di Bari,
    via Amendola 173, \phantom{aaz} 70126 Bari, Italy
      \newline \phantom{aaz} E-mail: \texttt{marco.zamparo@uniba.it}}}

\date{}

\begin{document}  
\maketitle

\begin{abstract}
We establish a sharp large deviation principle for renewal-reward
processes, supposing that each renewal involves a broad-sense reward
taking values in a real separable Banach space. In fact, we
demonstrate a weak large deviation principle without assuming any
exponential moment condition on the law of waiting times and rewards
by resorting to a sharp version of Cram\'er's theorem. We also exhibit
sufficient conditions for exponential tightness of renewal-reward
processes, which leads to a full large deviation principle.\\


\end{abstract}

\section{Main results}

{\it Renewal models} are widespread tools of probability that find
application in Queueing Theory \cite{AsmussenBook}, Insurance
\cite{DicksonBook}, Finance \cite{RSSTBook}, and Statistical Physics
\cite{GiacominBook} among others.  A renewal model describes some
event that occurs at the {\it renewal times} $T_1,T_2,\ldots$
involving the {\it rewards} $X_1,X_2,\ldots$ respectively. If
$S_1,S_2,\ldots$ denote the {\it waiting times} for a new occurrence
of the event, then the renewal time $T_i$ can be expressed for each
$i\ge 1$ in terms of the waiting times as
$T_i=S_1+\cdots+S_i$. Through this paper we assume that the waiting
time and reward pairs $(S_1,X_1),(S_2,X_2),\ldots$ form an independent
and identically distributed sequence of random vectors on a
probability space $(\Omega,\mathcal{F},\prob)$, the waiting times
taking positive real values and the rewards taking values in a real
separable Banach space $(\rew,\|{\cdot}\|)$ equipped with the Borel
$\sigma$-field $\Brew$.  Any dependence between $S_i$ and $X_i$ is
allowed and we can suppose without restriction that
$\lim_{i\uparrow\infty}T_i(\omega)=+\infty$ for all $\omega\in\Omega$.
The {\it cumulative reward} by the time $t\ge 0$ is the random
variable $W_t:=\sum_{i\ge 1}X_i\mathds{1}_{\{T_i\le t\}}$, which is
measurable because $\rew$ is separable \cite{Talagrand}.  The
stochastic process $t\mapsto W_t$ is the so-called {\it renewal-reward
  process} or {\it compound renewal process}, which plays an important
role in applications
\cite{AsmussenBook,DicksonBook,RSSTBook,GiacominBook}. The strong law
of large numbers holds for a renewal-reward process under the optimal
hypotheses $\Ex[S_1]<+\infty$ and $\Ex[\|X_1\|]<+\infty$, $\Ex$ being
expectation with respect to the law $\prob$, and can be proved by
combining standard arguments of renewal theory \cite{AsmussenBook}
with the classical strong law of large numbers of Kolmogorov in
separable Banach spaces \cite{Talagrand}.  This paper aims to
characterize the fluctuations of the cumulative reward $W_t$ as $t$
goes to infinity by means of large deviation bounds.

\subsection{Large deviation bounds}
\label{sec:mainres}

The \textit{Cram\'er's rate function} of waiting time
and reward pairs is the function $J$ that maps each
$(s,w)\in\Rl\times\rew$ in the extended real number
\begin{equation*}
J(s,w):=\sup_{(\zeta,\varphi)\in\Rl\times\rew^\star}\Big\{s\zeta+\varphi(w)-\ln\Ex\big[e^{\zeta S_1+\varphi(X_1)}\big]\Big\}.
\end{equation*}
Hereafter $\rew^\star$ denotes the topological dual of $\rew$, which
is understood as a Banach space with the norm induced by
$\|{\cdot}\|$.  In this paper a special role is played by the function
$\inf_{\gamma>0}\{\gamma J(\cdot/\gamma,\cdot/\gamma)\}$, whose
lower-semicontinuous regularization $\Upsilon$ associates every
$(\beta,w)\in\Rl\times\rew$ with
\begin{equation*}
  \Upsilon(\beta,w):=\adjustlimits\lim_{\delta\downarrow 0}\inf_{s\in (\beta-\delta,\beta+\delta)}\inf_{v\in B_{w,\delta}}\inf_{\gamma>0}
  \big\{\gamma J(s/\gamma,v/\gamma)\big\},
\end{equation*}
$B_{w,\delta}:=\{v\in\rew:\|v-w\|<\delta\}$ being the open ball of
center $w$ and radius $\delta$. Setting
$\elli:=-\liminf_{s\uparrow+\infty}(1/s)\ln\prob[S_1>s]$ and
$\ells:=-\limsup_{s\uparrow+\infty}(1/s)\ln\prob[S_1>s]$ and observing
that $0\le\ells\le\elli\le+\infty$, we make use of $\Upsilon$ to build
two rate functions $\Ii$ and $\Is$ on $\rew$ according to the formulas
\begin{equation*}
  \Ii:=\begin{cases}
  \inf_{\beta\in[0,1]}\{\Upsilon(\beta,\cdot\,)+(1-\beta)\elli\} & \mbox{if }\elli<+\infty,\\
  \Upsilon(1,\cdot\,) & \mbox{if }\elli=+\infty
  \end{cases}
\end{equation*}
and
\begin{equation*}
  \Is:=\begin{cases}
  \inf_{\beta\in[0,1]}\{\Upsilon(\beta,\cdot\,)+(1-\beta)\ells\} & \mbox{if }\ells<+\infty,\\
  \Upsilon(1,\cdot\,) & \mbox{if }\ells=+\infty.
  \end{cases}
\end{equation*}
The rate functions $\Ii$ and $\Is$ enter into a lower large deviation
bound and an upper large deviation bound, respectively, as stated by
the following theorem which collects the main results of the paper.
\begin{theorem}
\label{mainth}
The following conclusions hold:
\begin{enumerate}[{\upshape(a)}]
\item the rate functions $\Ii$ and $\Is$ are lower
  semicontinuous and convex;
\item if $G\subseteq\rew$ is open, then
\begin{equation*}
\liminf_{t\uparrow+\infty}\frac{1}{t}\ln\prob\bigg[\frac{W_t}{t}\in G\bigg]\ge -\inf_{w\in G}\big\{\Ii(w)\big\};
\end{equation*}
\item if $F\subseteq\rew$ is compact, then
\begin{equation*}
\limsup_{t\uparrow+\infty}\frac{1}{t}\ln\prob\bigg[\frac{W_t}{t}\in F\bigg]\le -\inf_{w\in F}\big\{\Is(w)\big\};
\end{equation*}

\item if $F\in\Brew$ is open convex, closed convex, or just convex
  when $\rew$ is finite-dimensional, then the bound of part {\upshape
    (c)} is valid whenever $\ells<+\infty$ or $\Is(0)<+\infty$;

\item if $\rew$ is finite-dimensional and $\Ex[e^{\zeta
    S_1+\sigma\|X_1\|}]<+\infty$ for some numbers $\zeta\le 0$ and
  $\sigma>0$, then $\Ii$ and $\Is$ have compact level sets and the
  bound of part {\upshape (c)} is valid for any closed set $F$;

\item if $\rew$ is infinite-dimensional and $\Ex[e^{\sigma
    S_1+\sigma\|X_1\|}]<+\infty$ for all $\sigma>0$, then
  $\Ii=\Is=\Upsilon(1,\cdot\,)$, $\Upsilon(1,\cdot\,)$ has compact
  level sets, and the bound of part {\upshape (c)} is valid for any
  closed set $F$.
\end{enumerate}
\end{theorem}

Theorem \ref{mainth} is proved in section \ref{sec:mainproof}. When
$\Ii=\Is$ the theorem establishes, through the lower large deviation
bound for open sets of part (b) and the upper large deviation bound
for compact sets of part (c), a {\it weak large deviation principle}
with {\it rate function} $\Ii=\Is$ for the renewal-reward process
$t\mapsto W_t$. We refer to \cite{DemboBook} for the language of large
deviation theory.  Part (d) states that the upper large deviation
bound also holds for open and closed convex sets provided that
$\ells<+\infty$ or $\Is(0)<+\infty$. It fails in general when
$\ells=+\infty$ and $\Is(0)=+\infty$, as we shall show in section
\ref{sec:mainproof} by means of two examples. We stress that no
assumption on the law of waiting time and reward pairs is made to
deduce parts (a), (b), (c), and (d) of theorem \ref{mainth}. Some
assumption is instead necessary for exponential tightness of the
distribution of the scaled cumulative reward $W_t/t$, which leads to a
full large deviation principle where the large deviation upper bound
is valid for all closed sets, and not only for those that are compact.
If $\Ii=\Is$ and $\rew$ is finite-dimensional, then part (e)
establishes a {\it full large deviation principle} with {\it good rate
  function} $\Ii=\Is$ under the exponential moment condition
$\Ex[e^{\zeta S_1+\sigma\|X_1\|}]<+\infty$ for some numbers $\zeta\le
0$ and $\sigma>0$. We recall that a rate function is ``good'' when it
has compact level sets.  Part (f) states that the same is true when
$\rew$ is infinite-dimensional and $\Ex[e^{\sigma
    S_1+\sigma\|X_1\|}]<+\infty$ for all $\sigma>0$. Obviously, we
have $\Ii=\Is$ if $\elli=\ells$ as expected in most applications. We
have $\Ii=\Is$ even if $\elli>\ells$ but rewards are dominated by
waiting times according to the following proposition, whose proof is
reported in appendix \ref{proof:ratewait}.
\begin{proposition}
\label{prop:ratewait}
Assume that there exists a positive real function $f$ on $[0,+\infty)$
  such that $\lim_{s\uparrow+\infty}f(s)/s=0$ and $\|X_1\|\le f(S_1)$
  with full probability. Then $\Ii=\Is=\Upsilon(1,\cdot\,)$.
\end{proposition}

\subsection{Discussion}
\label{sec:discussion}

Large deviation principles (LDPs) for renewal-reward processes have
been investigated by many authors over the past decades. Their
attention has been focused mostly on rewards taking real values and an
almost omnipresent hypothesis of previous works is the Cram\'er
condition on the law of waiting time and reward pairs: $\Ex[e^{\sigma
    S_1+\sigma\|X_1\|}]<+\infty$ for some number $\sigma>0$.

The simplest example of renewal-reward process has unit rewards and
corresponds to the counting renewal process $t\mapsto N_t:=\sum_{i\ge
  1}\mathds{1}_{\{T_i\le t\}}$. Glynn and Whitt \cite{Glynn1994}
investigated the connection between LDPs of the inverse processes
$t\mapsto N_t$ and $i\mapsto T_i$, providing a full LDP for $N_t$
under the Cram\'er condition. This condition was later relaxed by
Duffield and Whitt \cite{Duffield1998}. Jiang \cite{Jiang1994} studied
the large deviations of the extended counting renewal process
$t\mapsto\sum_{i\ge 1}\mathds{1}_{\{T_i\le i^\alpha t\}}$ with
$\alpha\in[0,1)$ under the Cram\'er condition.  Glynn and Whitt
  \cite{Glynn1994} and Duffield and Whitt \cite{Duffield1998},
  together with Puhalskii and Whitt \cite{Puhalskii1997}, also
  investigated the connection between sample-path LDPs of the
  processes $t\mapsto N_t$ and $i\mapsto T_i$ under the Cram\'er
  condition.

Starting from sample-path LDPs of inverse and compound processes,
Duffy and Rodgers-Lee \cite{Duffy2004} sketched a full LDP for
renewal-reward processes with real rewards by means of the contraction
principle under the stringent exponential moment condition
$\Ex[e^{\sigma S_1+\sigma\|X_1\|}]<+\infty$ for all $\sigma>0$. Some
full LDPs for real renewal-reward processes were later proposed by
Macci \cite{Macci2005,Macci2007} under existence and essentially
smoothness of the scaled cumulant generating function, which allow for
an application of the G\"artner-Ellis theorem \cite{DemboBook}.
Essentially smoothness of the scaled cumulant generating function has
been recently relaxed by Borovkov and Mogulskii
\cite{Borovkov2015,Borovkov2019}, which used the Cram\'er's theorem
\cite{DemboBook} to establish a full LDP under the Cram\'er
condition. Under this condition, they
\cite{Borovkov2016_1,Borovkov2016_2,Borovkov2016_3} have also obtained
sample-path LDPs for real renewal-reward processes.

A different approach based on empirical measures has been considered
by Lefevere, Mariani, and Zambotti \cite{Lefevere2011}, which have
investigated large deviations for the empirical measures of forward
and backward recurrence times associated with a renewal process, and
have then derived by contraction a full LDP for renewal-rewards
processes with rewards determined by the waiting times: $X_i:=f(S_i)$
for each $i$ with a bounded and continuous real function $f$. Later,
Mariani and Zambotti \cite{Mariani2014} have developed a renewal
version of Sanov's theorem by studying the empirical law of rewards
that take values in a generic Polish space under the hypothesis
$\Ex[e^{\sigma S_1}]<+\infty$ for all $\sigma>0$. By appealing to the
contraction principle, this result could give a full LDP for a 
renewal-reward process with rewards valued in a separable
Banach space provided that the exponential moment condition
$\Ex[e^{\sigma\|X_1\|}]<+\infty$ is satisfied for all $\sigma>0$ as
discussed by Schied \cite{Schied1998}.

These works leave open the question of whether some LDPs free from
exponential moment conditions can be established for renewal-reward
processes, in the wake of the sharp version of Cram\'er's theorem
demonstrated by Bahadur and Zabell \cite{BaZa}.  In a recent paper,
the author \cite{Marco} has dropped the Cram\'er condition in the
discrete-time framework, whereby waiting times have a lattice
distribution, by establishing a weak LDP for cumulative rewards free
from hypotheses.  The discrete-time framework is special because
allows a super-multiplicativity property of the probability that
$W_t/t$ belongs to a convex set to emerge by conditioning on the event
that the time $t$ is a renewal time. This super-multiplicativity
property was the key to get at sharp LDPs in
\cite{Marco}. Unfortunately, the same strategy does not extend to
waiting times with non-lattice distribution since conditioning on the
event that a certain time is a renewal time is not a meaningful
procedure in this case. The present paper overcomes the difficulty to
deal with general waiting times by making a better use of Cram\'er's
theorem than Borovkov and Mogulskii
\cite{Borovkov2015,Borovkov2019}. In fact, starting from the
Cram\'er's theory for waiting time and reward pairs, here we establish
a weak LDP for the renewal-reward process $t\mapsto W_t$ with no
restriction on waiting times and without assuming that the Cram\'er
condition is satisfied. Moreover, when finite-dimensional rewards are
considered, we provide a full LDP under the exponential moment
condition $\Ex[e^{\zeta S_1+\sigma\|X_1\|}]<+\infty$ for some numbers
$\zeta\le 0$ and $\sigma>0$, which is weaker than the Cram\'er
condition $\Ex[e^{\sigma S_1+\sigma\|X_1\|}]<+\infty$ for some
$\sigma>0$.  For instance, rewards that define macroscopic observables
in applications to Statistical Physics \cite{Models} are of the order
of magnitude of waiting times and always satisfy our weak exponential
moment condition, whereas in general they do not fulfill the Cram\'er
condition. But after all is said and done, a super-multiplicativity
argument is still the key, as it underlies the sharp version of
Cram\'er's theorem we have exploited to reach our results.

To conclude, we point out that, at variance with Borovkov and
Mogulskii, we propose optimal lower and upper large deviation bounds
with possibly different rate functions in order to even address
situations where the tail of the waiting time distribution is very
oscillating. For instance, a physical renewal model giving rise to two
possibly different rate functions has been found by Lefevere, Mariani,
and Zambotti \cite{Lefevere2011_2,Lefevere2012} in the description of
a free particle interacting with a heat bath.

\section{Proof of theorem \ref{mainth}}
\label{sec:mainproof}

The proof of theorem \ref{mainth} is organized as follows.  In section
\ref{sec:ratefunctions} we discuss some properties of $\Upsilon$ and
prove lower semicontinuity and convexity of the rate functions,
verifying part (a) of theorem \ref{mainth}. Section \ref{sec:opensets}
demonstrates the lower large deviation bound for open sets, thus
proving part (b) of theorem \ref{mainth}. The upper large deviation
bound for convex sets, that is part (d) of theorem \ref{mainth}, is
proved in section \ref{sec:convexsets}. In this section we also
exhibit examples with open and closed convex sets that demonstrate how
relaxing the hypotheses $\ells<+\infty$ and $\Is(0)<+\infty$ at the
same time leads to violate this bound.  Part (d) is used to prove part
(c) in section \ref{sec:compactsets}. Finally, the proof of parts (e)
and (f) of theorem \ref{mainth} are reported in section
\ref{sec:closedsets}.  The elements of Cram\'er's theory for waiting
time and reward pairs on which the proof is based are collected in
appendix \ref{Cramer}.

\subsection{Rate functions}
\label{sec:ratefunctions}

The function $\Upsilon$ satisfies the following properties, which will
be used in the sequel.
\begin{lemma}
  \label{lem:Upsilon_prop}
The following conclusions hold:
\begin{enumerate}[(i)]
\item $\Upsilon$ is lower semicontinuous and  convex;
\item $\Upsilon(\beta,w)\ge 0$ and
  $\Upsilon(a\beta,aw)=a\Upsilon(\beta,w)$ for all $\beta\in\Rl$,
  $w\in\rew$, and $a>0$;
\item $\Upsilon(0,0)=0$ and $\Upsilon(\beta,w)=+\infty$ for all $\beta<0$ and $w\in\rew$;
\item for every $\beta>0$ and $w\in\rew$
\begin{equation*}
  \Upsilon(\beta,w)=\adjustlimits\lim_{\delta\downarrow 0}\inf_{v\in B_{w,\delta}}\inf_{\gamma>0}\big\{\gamma J(\beta/\gamma,v/\gamma)\big\}.
\end{equation*}
\end{enumerate}
\end{lemma}

\begin{proof}
As $\Upsilon$ is lower semicontinuous by construction, in order to
prove part $(i)$ it suffices to verify convexity. We show that, for
every given integer $k\ge 1$, the function
$\inf_{\gamma\in[1/k,k]}\{\gamma J(\cdot/\gamma,\cdot/\gamma)\}$ over
$\Rl\times\rew$ is the convex conjugate of a certain function $F_k$ on
$\Rl\times\rew^\star$. This way, $\inf_{\gamma>0}\{\gamma
J(\cdot/\gamma,\cdot/\gamma)\}=\lim_{k\uparrow\infty}\inf_{\gamma\in[1/k,k]}\{\gamma
J(\cdot/\gamma,\cdot/\gamma)\}$ is convex, and so is $\Upsilon$. Pick
$s\in\Rl$ and $w\in\rew$ and denote by $\mathcal{D}$ the set
$\{(\zeta,\varphi)\in\Rl\times\rew^\star:\Ex[e^{\zeta
    S_1+\varphi(w)}]<+\infty\}$.  The function that associates
$(\zeta,\varphi)\in\mathcal{D}$ with $\ln\Ex[e^{\zeta
    S_1+\varphi(X_1)}]$ is lower semicontinuous by Fatou's lemma and
convex, so that the real function that maps
$(\gamma,\zeta,\varphi)\in[1/k,k]\times\mathcal{D}$ in
$\varphi(w)-s\zeta-\gamma\ln\Ex[e^{\zeta S_1+\varphi(X_1)}]$ is
concave and upper semicontinuous with respect to $(\zeta,\varphi)$ for
each fixed $\gamma\in[1/k,k]$ and convex and continuous with respect
to $\gamma$ for each fixed pair $(\zeta,\varphi)\in\mathcal{D}$. Then,
the compactness of the interval $[1/k,k]$ allows an application of
Sion's minimax theorem to get
\begin{align}
  \nonumber
  \inf_{\gamma\in[1/k,k]}\big\{\gamma J(s/\gamma,w/\gamma)\big\}&=\adjustlimits\inf_{\gamma\in[1/k,k]}\sup_{(\zeta,\varphi)\in\mathcal{D}}\Big\{s\zeta+\varphi(w)-
  \gamma\ln\Ex\big[e^{\zeta S_1+\varphi(X_1)}\big]\Big\}\\
  \nonumber
  &=\adjustlimits\sup_{(\zeta,\varphi)\in\mathcal{D}}\inf_{\gamma\in[1/k,k]}\Big\{s\zeta+\varphi(w)-
  \gamma\ln\Ex\big[e^{\zeta S_1+\varphi(X_1)}\big]\Big\}\\
  \nonumber
  &=\sup_{(\zeta,\varphi)\in\Rl\times\rew^\star}\Big\{s\zeta+\varphi(w)-F_k(\zeta,\varphi)\Big\}
\end{align}
with
\begin{equation*}
F_k(\zeta,\varphi):=\max\Big\{(1/k)\ln\Ex\big[e^{\zeta S_1+\varphi(X_1)}\big],k\ln\Ex\big[e^{\zeta S_1+\varphi(X_1)}\big]\Big\}
\end{equation*}
for all $\zeta\in\Rl$ and $\varphi\in\rew^\star$.

Let us move to part $(ii)$. We have $J(s,w)\ge 0$ for all $s\in\Rl$
and $w\in\rew$ because $s\zeta+\varphi(w)-\ln\Ex[e^{\zeta
    S_1+\varphi(X_1)}]=0$ when $\zeta=0$ and $\varphi=0$. Then,
$\Upsilon(\beta,w)\ge 0$ for all $\beta$ and $w$. The property
$\Upsilon(a\beta,aw)=a\Upsilon(\beta,w)$ for all $\beta\in\Rl$,
$w\in\rew$, and $a>0$ is immediate.

Regarding part $(iii)$, let us observe that
$\lim_{\zeta\downarrow-\infty}\{s\zeta+\varphi(w)-\ln\Ex[e^{\zeta
    S_1+\varphi(X_1)}]\}=+\infty$ for each $s<0$ and
$\varphi\in\rew^\star$. This way, $J(s,w)=+\infty$ for each $s<0$ and
$w\in\rew$, which yields $\Upsilon(\beta,w)=+\infty$ for every
$\beta<0$ and $w\in\rew$. As far as the equality $\Upsilon(0,0)=0$ is
concerned, in the light of part $(ii)$ it remains to demonstrate that
$\Upsilon(0,0)\le 0$.  The function that maps
$(\zeta,\varphi)\in\Rl\times\rew^\star$ with $\ln\Ex[e^{\zeta
    S_1+\varphi(X_1)}]$ is proper convex, so that there exist
$s_o\in\Rl$, $w_o\in\rew$, and a constant $c$ such that
$\ln\Ex[e^{\zeta S_1+\varphi(X_1)}]\ge s_o\zeta+\varphi(w_o)-c$ for
all $\zeta$ and $\varphi$ (see \cite{Zalinescu}, theorem 2.2.6). It
follows that $J(s_o,w_o)\le c<+\infty$. Let $a$ be a small positive
number such that $ a s_o\in(-1,1)$ and $a w_o\in B_{0,1}$. Then, for
all $\delta>0$ we find the bound
\begin{equation*}
  \adjustlimits\inf_{s\in (-\delta,\delta)}\inf_{v\in B_{0,\delta}}\inf_{\gamma>0}
  \big\{\gamma J(s/\gamma,v/\gamma)\big\}\le\inf_{\gamma>0}
  \big\{\gamma J(a\delta s_o/\gamma,a\delta w_o/\gamma)\big\}\le a\delta J(s_o,w_o),
\end{equation*}
which gives $\Upsilon(0,0)\le 0$ once $\delta$ is sent to 0.

To conclude, let us prove part $(iv)$.  Fix $\beta>0$ and
$w\in\rew$. It is clear that
  \begin{equation*}
  \Upsilon(\beta,w)\le\adjustlimits\lim_{\delta\downarrow 0}\inf_{v\in B_{w,\delta}}\inf_{\gamma>0}\big\{\gamma J(\beta/\gamma,v/\gamma)\big\}.
\end{equation*}
  Let us demonstrate the opposite bound. For all $\delta\in(0,\beta)$
  and $s\in(\beta-\delta,\beta+\delta)$ we have $s>0$ and $B_{\beta
    w/s,\beta\delta/s}\subseteq
  B_{w,(\beta+\|w\|)\delta/(\beta-\delta)}$. The latter is due to the
  fact that if $s\in(\beta-\delta,\beta+\delta)$ and $v\in B_{\beta
    w/s,\beta\delta/s}$, then
  \begin{align}
    \nonumber
    \|v-w\|&\le \|v-\beta w/s\|+\|\beta w/s-w\|<\beta\delta/s+|\beta/s-1|\|w\|\\
    \nonumber
    &<\beta\delta/(\beta-\delta)+\delta\|w\|/(\beta-\delta)=(\beta+\|w\|)\delta/(\beta-\delta).
\end{align}
Thus, recalling that $J$ is non-negative, for every
$\delta\in(0,\beta)$ we can write
\begin{align}
  \nonumber
  \adjustlimits\inf_{s\in (\beta-\delta,\beta+\delta)}\inf_{v\in B_{w,\delta}}\inf_{\gamma>0}\big\{\gamma J(s/\gamma,v/\gamma)\big\}&=
  \adjustlimits\inf_{s\in (\beta-\delta,\beta+\delta)}\inf_{v\in B_{\beta w/s,\beta\delta/s}}\inf_{\gamma>0}\big\{(\gamma J(s/\gamma,sv/\gamma\beta)\big\}\\
  \nonumber
  &=\adjustlimits\inf_{s\in (\beta-\delta,\beta+\delta)}\inf_{v\in B_{\beta w/s,\beta\delta/s}}\inf_{\gamma>0}\big\{(s\gamma/\beta) J(\beta/\gamma,v/\gamma)\big\}\\
  \nonumber
  &\ge (1-\delta/\beta)\inf_{v\in B_{w,(\beta+\|w\|)\delta/(\beta-\delta)}}\inf_{\gamma>0}\big\{\gamma J(\beta/\gamma,v/\gamma)\big\}.
\end{align}
This inequality shows that
\begin{equation*}
  \Upsilon(\beta,w)\ge\adjustlimits\lim_{\delta\downarrow 0}\inf_{v\in B_{w,\delta}}\inf_{\gamma>0}\big\{\gamma J(\beta/\gamma,v/\gamma)\big\}.
\qedhere
\end{equation*}
\end{proof}

We are now in the position to prove part (a) of theorem \ref{mainth}.
\begin{proposition}
The rate functions $\Ii$ and $\Is$ are lower semicontinuous and convex.
\end{proposition}

\begin{proof}
We address the rate function $\Ii$. The same arguments apply to $\Is$.
If $\elli=+\infty$, then $\Ii=\Upsilon(1,\cdot\,)$ and lower
semicontinuity and convexity of $\Ii$ immediately follow from part
$(i)$ of lemma \ref{lem:Upsilon_prop}. Assume that $\elli<+\infty$.
In order to demonstrate lower semicontinuity of $\Ii$, let us show
that the set $F:=\{w\in\rew:\Ii(w)\le\lambda\}$ is closed for any
given real number $\lambda$. Let $\{w_k\}_{k\ge 1}$ be a sequence in
$F$ converging to a point $w$. We claim that $w\in F$. In fact, by the
lower semicontinuity of $\Upsilon$ and the compactness of $[0,1]$,
for each $k\ge 1$ there exists $\beta_k\in[0,1]$ such that
$\Ii(w_k)=\Upsilon(\beta_k,w_k)+(1-\beta_k)\elli$. The compactness of
$[0,1]$ also entails that there exists a subsequence
$\{\beta_{k_j}\}_{j\ge 1}$ that converges to some number
$\beta_o\in[0,1]$. We have $\lambda\ge\Ii(w_{k_j})=
\Upsilon(\beta_{k_j},w_{k_j})+(1-\beta_{k_j})\elli$ for all $j\ge 1$,
which gives $\lambda\ge \Upsilon(\beta_o,w)+(1-\beta_o)\elli\ge
\Ii(w)$ once $j$ is sent to infinity. Thus, $w\in F$.

As far as convexity of $\Ii$ is concerned, given $w_1\in\rew$ and
$w_2\in\rew$, let $\beta_1\in[0,1]$ and $\beta_2\in[0,1]$ be such that
$\Ii(w_1)=\Upsilon(\beta_1,w_1)+(1-\beta_1)\elli$ and
$\Ii(w_2)=\Upsilon(\beta_2,w_2)+(1-\beta_2)\elli$. We recall that the
existence of $\beta_1$ and $\beta_2$ is guaranteed by the lower
semicontinuity of $\Upsilon$. This way, if $a_1\ge 0$ and $a_2\ge 0$
are two numbers such that $a_1+a_2=1$, then convexity of $\Upsilon$
shows that
\begin{align}
  \nonumber
  \Ii(a_1w_1+a_2w_2)&\le\Upsilon(a_1\beta_1+a_2\beta_2,a_1w_1+a_2w_2)+(1-a_1\beta_1-a_2\beta_2)\elli\\
  \nonumber
  &\le a_1\big[\Upsilon(\beta_1,w_1)+(1-\beta_1)\elli\big]+a_2\big[\Upsilon(\beta_2,w_2)+(1-\beta_2)\elli\big]\\
  \nonumber
  &= a_1\Ii(w_1)+a_2\Ii(w_2).
  \qedhere
\end{align}
\end{proof}

\subsection{The lower large deviation bound for open sets}
\label{sec:opensets}

The proof of part (b) of theorem \ref{mainth} relies on the following
lower bound: for each set $A\in\Brew$ and integers $1\le p<q$ 
\begin{align}
  \nonumber
  \prob\bigg[\frac{W_t}{t}\in A\bigg]&\ge\prob\bigg[\frac{W_t}{t}\in A, \,T_p\le t<T_q\bigg]\\
  &=\sum_{n=p}^{q-1}\prob\Bigg[\frac{1}{t}\sum_{i=1}^nX_i\in A, \,T_n\le t<T_{n+1}\Bigg].
  \label{eq:partenza}
\end{align}
This lower bound gives the forthcoming lemma, which applies for both
$\elli<+\infty$ and $\elli=+\infty$ and demonstrates part (b) of
theorem \ref{mainth} directly when $\elli=+\infty$.
\begin{lemma}
\label{lower_bound_inf}
  For every $G\subseteq\rew$ open and $w\in G$
\begin{equation*}
  \liminf_{t\uparrow+\infty}\frac{1}{t}\ln\prob\bigg[\frac{W_t}{t}\in G\bigg]\ge-\Upsilon(1,w).
\end{equation*}
\end{lemma}

\begin{proof}
  Pick an open set $G$ in $\rew$. We shall prove that for each point
  $w\in G$ and real number $\gamma>0$
\begin{equation}
  \label{lower_bound_inf_1}
  \liminf_{t\uparrow+\infty}\frac{1}{t}\ln\prob\bigg[\frac{W_t}{t}\in G\bigg]\ge-\gamma J(1/\gamma,w/\gamma).
\end{equation}
This bound yields the lemma by optimizing over $\gamma$ and by
invoking part $(iv)$ of lemma \ref{lem:Upsilon_prop}.

Fix an arbitrary point $w\in G$ and an arbitrary real number
$\gamma>0$.  As $G$ is open, there exists $\delta>0$ such that
$B_{w,2\delta}\subseteq G$.  Since $\prob[S_1>0]=1$ there exist a
small number $m>0$ and a large number $M>0$ with the property
$\prob[S_1\ge m,\,\|X_1\|\le M]\ge 1/2$.  Let $\epsilon\in(0,1)$ be
such that $(1+2M/m+\|w\|)\epsilon<\delta$ and let $t_o>0$ be such that
$\gamma(1-\epsilon)t_o\ge 1$, $m\le 2\epsilon t_o$,
$\|w\|<\delta\gamma t_o$, and $\epsilon^2t_o\ge m+1/\gamma$.  Set
$p_t:=\lfloor \gamma(1-\epsilon)t\rfloor$ and $q_t:=p_t+\lfloor
2\epsilon t/m\rfloor$. For $t>t_o$ we have $p_t\ge 1$ as
$\gamma(1-\epsilon)t>\gamma(1-\epsilon)t_o\ge 1$ and $p_t<q_t$ as
$2\epsilon t/m>2\epsilon t_o/m\ge 1$. For brevity, we denote by
$\mathbb{M}_t$ the probability measure that maps a set $A\in\Brew$ in
\begin{equation*}
\mathbb{M}_t[A]:=\prob\bigg[A\,\bigg|\min_{p_t<i\le q_t}\big\{S_i\big\}\ge m,\,\max_{p_t<i\le q_t}\big\{\|X_i\|\big\}\le M\bigg],
\end{equation*}
and we observe that
\begin{align}
  \nonumber
  \prob[A]&\ge \mathbb{M}_t[A]\cdot\prob\bigg[\min_{p_t<i\le q_t}\big\{S_i\big\}\ge m,\,\max_{p_t<i\le q_t}\big\{\|X_i\|\big\}\le M\bigg]\\
  \nonumber
  &=\mathbb{M}_t[A]\cdot\prob\bigg[S_1>m,\,\|X_1\|\le M\bigg]^{2\epsilon t/m}\ge \frac{1}{4^{\epsilon t/m}}\,\mathbb{M}_t[A].
\end{align}

Bound (\ref{eq:partenza}) gives for any $t>t_o$
\begin{align}
  \nonumber
  \prob\bigg[\frac{W_t}{t}\in G\bigg]&\ge\sum_{n=p_t}^{q_t-1}\prob\Bigg[\frac{1}{t}\sum_{i=1}^nX_i\in B_{w,2\delta}, \,T_n\le t<T_{n+1}\Bigg]\\
  \nonumber
  &\ge\frac{1}{4^{\epsilon t/m}}\sum_{n=p_t}^{q_t-1}\mathbb{M}_t\Bigg[\frac{1}{t}\sum_{i=1}^nX_i\in B_{w,2\delta}, \,T_n\le t<T_{n+1}\Bigg].
\end{align}
We now notice that the condition $(1/p_t)\sum_{i=1}^{p_t}X_i\in
B_{w/\gamma,\epsilon/\gamma}$ implies $(1/t)\sum_{i=1}^nX_i\in
B_{w,2\delta}$ for each $n\in[p_t,q_t)$ when $t>t_o$ and
  $\max_{p_t<i\le q_t}\{\|X_i\|\}\le M$. In fact, recalling that
  $(1+2M/m+\|w\|)\epsilon<\delta$ and that $\|w\|<\delta\gamma t_o$,
  for $t>t_o$ we find
\begin{align}
  \nonumber
  \Bigg\|\frac{1}{t}\sum_{i=1}^nX_i-w\Bigg\|&=\Bigg\|\frac{p_t}{t}\bigg(\frac{1}{p_t}\sum_{i=1}^{p_t}X_i-\frac{w}{\gamma}\bigg)+
  \frac{1}{t}\sum_{i=p_t+1}^nX_i+\bigg(\frac{p_t}{\gamma t}-1\bigg)w\Bigg\|\\
\nonumber
&<\frac{p_t\epsilon}{\gamma t}+\frac{M(n-p_t)}{t}+\bigg|\frac{p_t}{\gamma t}-1\bigg|\|w\|
\le(1+2M/m+\|w\|)\epsilon+\frac{\|w\|}{\gamma t_o}<2\delta.
\end{align}
This argument yields for $t>t_o$
\begin{align}
  \nonumber
  \prob\bigg[\frac{W_t}{t}\in G\bigg]&\ge\frac{1}{4^{\epsilon t/m}}
  \sum_{n=p_t}^{q_t-1}\mathbb{M}_t\Bigg[\frac{1}{p_t}\sum_{i=1}^{p_t}X_i\in B_{w/\gamma,\epsilon/\gamma}, \,T_n\le t<T_{n+1}\Bigg]\\
  \nonumber
  &=\frac{1}{4^{\epsilon t/m}}\,\mathbb{M}_t\Bigg[\frac{1}{p_t}\sum_{i=1}^{p_t}X_i\in B_{w/\gamma,\epsilon/\gamma},\,T_{p_t}\le t<T_{q_t}\Bigg].
\end{align}

The condition $T_{p_t}<(1/\gamma+\epsilon/\gamma)p_t$ implies
$T_{p_t}\le t$. Moreover, under the constraints $t>t_o$ and
$\min_{p_t<i\le q_t}\{S_i\}\ge m$, the condition
$T_{p_t}>(1/\gamma-\epsilon/\gamma)p_t$ entails $t<T_{q_t}$. Indeed,
since $\epsilon^2t_o\ge m+1/\gamma$ by construction we have
\begin{align}
\nonumber
T_{q_t}&=T_{p_t}+\sum_{i=p_t+1}^{q_t}S_i>(1/\gamma-\epsilon/\gamma)p_t+(q_t-p_t)m\\
\nonumber
&>(1/\gamma-\epsilon/\gamma)[\gamma(1-\epsilon)t-1]+(2\epsilon t/m-1)m\\
\nonumber
&=t+\epsilon^2t-m-(1-\epsilon)/\gamma>t+\epsilon^2t_o-m-1/\gamma\ge t.
\end{align}
It follows that for $t>t_o$
\begin{align}
  \nonumber
  \prob\bigg[\frac{W_t}{t}\in G\bigg]&\ge
  \frac{1}{4^{\epsilon t/m}}
  \,\mathbb{M}_t\Bigg[\frac{1}{p_t}\sum_{i=1}^{p_t}X_i\in B_{w/\gamma,\epsilon/\gamma},\,(1/\gamma-\epsilon/\gamma)p_t<T_{p_t}<(1/\gamma+\epsilon/\gamma)p_t\Bigg]\\
\nonumber
  &=\frac{1}{4^{\epsilon t/m}}
  \,\prob\Bigg[\frac{1}{p_t}\sum_{i=1}^{p_t}(S_i,X_i)\in (1/\gamma-\epsilon/\gamma,1/\gamma+\epsilon/\gamma)\times B_{w/\gamma,\epsilon/\gamma}\Bigg],
\end{align}
where the last equality is due to the fact that
$\sum_{i=1}^{p_t}(S_i,X_i)$ is independent of $\min_{p_t<i\le
  q_t}\{S_i\}$ and $\max_{p_t<i\le q_t}\{\|X_i\|\}$. At this point,
part $(ii)$ of proposition \ref{WLDP} allows us to conclude that
\begin{equation*}
  \liminf_{t\uparrow+\infty}\frac{1}{t}\ln\prob\bigg[\frac{W_t}{t}\in G\bigg]
  \ge-\gamma(1-\epsilon)J(1/\gamma,w/\gamma)-\frac{\epsilon}{m}\ln 4.
\end{equation*}
We get (\ref{lower_bound_inf_1}) from here since $\epsilon\in(0,1)$ is
any number smaller than $\delta/(1+2M/m+\|w\|)$.
\end{proof}

The next lemma proves part (b) of theorem \ref{mainth} when
$\elli<+\infty$.

\begin{lemma}
\label{lower_bound_fin}
Assume that $\elli<+\infty$.  For each $G\subseteq\rew$ open, $w\in G$, and $\beta\in[0,1]$
\begin{equation*}
  \liminf_{t\uparrow+\infty}\frac{1}{t}\ln\prob\bigg[\frac{W_t}{t}\in G\bigg]\ge-\Upsilon(\beta,w)-(1-\beta)\elli.
\end{equation*}
\end{lemma}

\begin{proof}
 The instance $\beta=1$ is solved by lemma \ref{lower_bound_inf}, so
 that we must tackle the case $\beta<1$. Given an open set $G$ in
 $\rew$, we prove that for each $w\in G$ and real numbers $\gamma>0$
 and $s<1$
\begin{equation}
  \label{lower_bound_fin_1}
  \liminf_{t\uparrow\infty}\frac{1}{t}\ln\prob\bigg[\frac{W_t}{t}\in G\bigg]\ge-\gamma J(s/\gamma,w/\gamma)-(1-s)\elli.
\end{equation}
The lemma follows from here by recalling the definition of $\Upsilon$
and by optimizing over $\gamma$, $s$, and $w$.

Fix $w\in G$, $\gamma>0$, and $s<1$. If $s<0$, then there is nothing
to prove because $J(s/\gamma,w/\gamma)=+\infty$ as we have seen in the
proof of part $(iii)$ of lemma \ref{lem:Upsilon_prop}. Assume that
$s\in[0,1)$ and pick a small number $\epsilon>0$ such that
  $B_{w,\epsilon}\subseteq G$ and $s+\gamma\epsilon\le 1$. Let
  $\delta>0$ and $t_o>0$ be two real numbers satisfying
  $\gamma\delta+\|w\|/(\gamma t_o)\le\epsilon$ and $\lfloor \gamma
  t_o\rfloor\ge 1$.  Set $p_t:=\lfloor \gamma t\rfloor$. For $t>t_o$
  we have $1\le p_t\le\gamma t$ and (\ref{eq:partenza}) gives
\begin{align}
  \nonumber
  \prob\bigg[\frac{W_t}{t}\in G\bigg]&\ge\prob\Bigg[\frac{1}{t}\sum_{i=1}^{p_t}X_i\in B_{w,\epsilon},\,T_{p_t}\le t<T_{p_t+1}\Bigg]\\
\nonumber
  &\ge\prob\Bigg[\frac{1}{t}\sum_{i=1}^{p_t}X_i\in B_{w,\epsilon}, \,\gamma \frac{T_{p_t}}{p_t}\le 1,\,S_{p_t+1}>t-T_{p_t}\Bigg].
\end{align}
Since $s+\gamma\epsilon\le 1$, we can write down the bound
\begin{align}
  \nonumber
  \prob\bigg[\frac{W_t}{t}\in G\bigg]
  &\ge\prob\Bigg[\frac{1}{t}\sum_{i=1}^{p_t}X_i\in B_{w,\epsilon},\, \gamma\frac{T_{p_t}}{p_t}\in(s-\gamma\epsilon,s+\gamma\epsilon),
    \,S_{p_t+1}>t-T_{p_t}\Bigg]\\
  \nonumber
  &\ge\prob\Bigg[\frac{1}{t}\sum_{i=1}^{p_t}X_i\in B_{w,\epsilon},\,\frac{T_{p_t}}{p_t}\in(s/\gamma-\epsilon,s/\gamma+\epsilon),
    \,S_{p_t+1}>t-(s/\gamma-\epsilon)p_t\Bigg]\\
  \nonumber
  &=\prob\Bigg[\frac{1}{t}\sum_{i=1}^{p_t}X_i\in B_{w,\epsilon},\,\frac{T_{p_t}}{p_t}\in(s/\gamma-\epsilon,s/\gamma+\epsilon)\Bigg]
  \cdot\prob\big[S_1>t-(s/\gamma-\epsilon)p_t\big].
\end{align}
We now observe that the condition $(1/p_t)\sum_{i=1}^{p_t}X_i\in
B_{w/\gamma,\delta}$ implies $(1/t)\sum_{i=1}^{p_t}X_i\in
B_{w,\epsilon}$ for $t>t_o$. Indeed, recalling that
$\gamma\delta+\|w\|/(\gamma t_o)\le\epsilon$, for $t>t_o$ we find
\begin{equation*}
  \Bigg\|\frac{1}{t}\sum_{i=1}^{p_t}X_i-w\Bigg\|=\frac{p_t}{t}\Bigg\|\frac{1}{p_t}\sum_{i=1}^{p_t}X_i-w/\gamma+w/\gamma-\frac{t}{p_t}w\Bigg\|
  <\gamma\delta+\|w\|/(\gamma t)\le\epsilon.
\end{equation*}
Then, for $t>t_o$ we have
\begin{align}
  \nonumber
  \prob\bigg[\frac{W_t}{t}\in G\bigg]
 &\ge\prob\Bigg[\frac{1}{p_t}\sum_{i=1}^{p_t}X_i\in B_{w/\gamma,\delta},\,\frac{T_{p_t}}{p_t}\in(s/\gamma-\epsilon,s/\gamma+\epsilon)\Bigg]
  \cdot\prob\big[S_1>t-(s/\gamma-\epsilon)p_t\big]\\
  \nonumber
  &=\prob\Bigg[\frac{1}{p_t}\sum_{i=1}^{p_t}(S_i,X_i)\in (s/\gamma-\epsilon,s/\gamma+\epsilon)\times B_{w/\gamma,\delta}\Bigg]
  \cdot\prob\big[S_1>t-(s/\gamma-\epsilon)p_t\big],
\end{align}
so that part $(ii)$ of proposition \ref{WLDP} and
$-\liminf_{\sigma\uparrow+\infty}(1/\sigma)\prob[S_1>\sigma]=:\elli$ yield
\begin{align}
  \nonumber
  \liminf_{t\uparrow+\infty}\frac{1}{t}\ln\prob\bigg[\frac{W_t}{t}\in G\bigg]&\ge
  -\gamma J(s/\gamma,w/\gamma)-(1-s)\elli-\gamma\epsilon\elli.
\end{align}
Since $\epsilon$ is any positive number small enough, this bound
demonstrates (\ref{lower_bound_fin_1}).
\end{proof}

\subsection{The upper large deviation bound for convex sets}
\label{sec:convexsets}

The starting point to prove some upper large deviation bounds is the
following inequality, which holds for every set $A\in\Brew$ and
integer $q>1$:
\begin{align}
  \nonumber
  \prob\bigg[\frac{W_t}{t}\in A\bigg]&\le\prob\bigg[\frac{W_t}{t}\in A,\,T_q>t\bigg]+\prob\big[T_q\le t\big]\\
&\le\mathds{1}_{\{0\in A\}}\prob[S_1>t]
  +\sum_{n=1}^{q-1}\prob\Bigg[\frac{1}{t}\sum_{i=1}^nX_i\in A,\,T_n\le t<T_{n+1}\Bigg]+\prob\big[T_q\le t\big].
  \label{eq:starting_upper}
\end{align}
We complement this inequality with two lemmas, the first of which
controls the small values of the waiting times.
\begin{lemma}
  \label{lem:aux1}
  There exists a real number $\kappa>0$ such that for all sufficiently
  large numbers $\gamma$ and $t$
  \begin{equation*}
  \prob\big[T_{\lfloor \gamma t\rfloor}\le t\big]\le e^{-\kappa\gamma t}.
\end{equation*}
\end{lemma}

\begin{proof}
Since $\prob[S_1>0]=1$, there exists a number $\eta>0$ such that
$\xi:=\prob[S_1\ge\eta]>0$. Pick three real numbers $\gamma\ge
3/\eta\xi$, $t\ge\eta\xi$, and $\lambda\ge 0$ and set $i:=\lfloor
\gamma t\rfloor$ for brevity. Chernoff bound and the equality
$e^{-z}\le 1-z+z^2/2$ valid for all $z\ge 0$ allows us to write down
the bound
\begin{align}
  \nonumber
  \prob\big[T_i\le t\big]&\le e^{\lambda t}\,\Ex[e^{-\lambda T_i}]\le
  e^{\lambda t}\,\Ex[e^{-\lambda S_1\wedge\eta}]^i\\
  \nonumber
  &\le e^{\lambda t}\Big\{1-\lambda\Ex\big[S_1\wedge\eta\big]+\lambda^2\Ex\big[(S_1\wedge\eta)^2\big]/2\Big\}^i
\le e^{\lambda t}\Big\{1-\lambda\eta\xi+\lambda^2\eta^2/2\Big\}^i.
\end{align}
At this point, the inequality $1+z\le e^z$ valid for all $z\in\Rl$
gives
\begin{equation}
  \prob\big[T_i\le t\big]\le e^{\lambda t-\lambda\eta\xi i+\lambda^2\eta^2i/2}.
\label{aux11}
\end{equation}
Since $\gamma\ge 3/\eta\xi$ and $t\ge\eta\xi$ we have
\begin{align}
  \nonumber
\eta\xi i-t\ge \eta\xi(\gamma t-1)-t=\gamma t[\eta\xi-1/\gamma(1+\eta\xi/t)]\ge\gamma t[\eta\xi-2/\gamma]\ge \gamma t\eta\xi/3>0.
\end{align}
This way, we can set $\lambda:=(\eta\xi i-t)/(\eta^2i)$ in
(\ref{aux11}) to get
\begin{equation*}
  \prob\big[T_i\le t\big]\le e^{-\frac{(\eta\xi i-t)^2}{2\eta^2 i}}\le e^{-\frac{\gamma^2 t^2\eta^2\xi^2}{18\eta^2\gamma t}}=e^{-\frac{\gamma\xi^2}{18}t}.
\qedhere
\end{equation*}
\end{proof}

The second lemma is more technical and is needed to estimate
probabilities involving convex sets.
\begin{lemma}
\label{lem:aux2}
Let $\alpha<\beta$ be two real numbers and let $C\subseteq\rew$ be
open convex, closed convex, or any convex set in $\Brew$ when $\rew$
is finite-dimensional.  Then, for all $n\ge 1$ and $t>0$
\begin{equation*}
  \ln \prob\Bigg[\frac{1}{t}\sum_{i=1}^n(S_i,X_i)\in[\alpha,\beta]\times C\Bigg]
  \le-t\inf_{(s,w)\in[\alpha,\beta]\times C}\big\{\Upsilon(s,w)\big\}.
\end{equation*}
\end{lemma}

\begin{proof}
Fix an integer $n\ge 1$ and a real number $t>0$. Set
$\alpha_o:=t\alpha/n$ and $\beta_o:=t\beta/n$ and denote by $C_o$ the
convex set $\{ tw/n:w\in C\}\in\Brew$. Then, $C_o$ is open or closed
if $C$ is open or closed and lemma \ref{lemma_conv} shows that
\begin{align}
  \nonumber
  \ln \prob\Bigg[\frac{1}{t}\sum_{i=1}^n(S_i,X_i)\in[\alpha,\beta]\times C\Bigg]&
  =\ln \prob\Bigg[\frac{1}{n}\sum_{i=1}^n(S_i,X_i)\in[\alpha_o,\beta_o]\times C_o\Bigg]\\
  \nonumber
  &\le-n\inf_{(s,w)\in[\alpha_o,\beta_o]\times C_o}\big\{J(s,w)\big\}\\
  \nonumber
  &=-t\inf_{(s,w)\in[\alpha,\beta]\times C}\big\{(n/t)J(ts/n,tw/n)\big\}.
\end{align}
On the other hand, we have $(n/t)J(ts/n,tw/n)\ge \Upsilon(s,w)$ for
all $s\in\Rl$ and $w\in\rew$ by definition.
\end{proof}

We are now in the position to demonstrate part (d) of theorem
\ref{mainth}. An upper large deviation bound for convex sets comes
from the following lemma.
\begin{lemma}
\label{lem:aux3}
Let $\ell\le\ells$ be a real number and let $C\subseteq\rew$ be
open convex, closed convex, or any convex set in $\Brew$ when $\rew$
is finite-dimensional.  Then
\begin{equation*}
  \limsup_{t\uparrow+\infty}\frac{1}{t}\ln\prob\bigg[\frac{W_t}{t}\in C\bigg]\le -\inf_{w\in C}\inf_{\beta\in[0,1]}\big\{\Upsilon(\beta,w)+(1-\beta)\ell\big\}.
\end{equation*}
\end{lemma}
If $\ells<+\infty$, then part (d) of theorem \ref{mainth} follows from
this lemma by making the choice $\ell=\ells$. If $\ells=+\infty$ and
$\Is(0)<+\infty$, then part (d) of theorem \ref{mainth} is obtained by
taking $\ell=\Is(0)=\Upsilon(1,0)$. While the former is manifest, the
latter is due to parts $(i)$ and $(ii)$ of lemma
\ref{lem:Upsilon_prop}, which show that for all $w\in\rew$ and
$\beta\in[0,1]$
\begin{align}
  \nonumber
  \Is(w):=\Upsilon(1,w)=\Upsilon\bigg(\frac{2\beta+2-2\beta}{2},\frac{2w+0}{2}\bigg)&\le\frac{1}{2}\Upsilon(2\beta,2w)+\frac{1}{2}\Upsilon(2-2\beta,0)\\
  \nonumber
  &=\Upsilon(\beta,w)+(1-\beta)\Upsilon(1,0).
\end{align}

\begin{proof}
  Pick a real number $\ell\le\ells$ and notice that $\inf_{w\in
    C}\inf_{\beta\in[0,1]}\{\Upsilon(\beta,w)+(1-\beta)\ell\}>-\infty$
  as $\Upsilon$ is non-negative by lemma \ref{lem:Upsilon_prop}. Fix
  real numbers $\lambda<\inf_{w\in
    C}\inf_{\beta\in[0,1]}\{\Upsilon(\beta,w)+(1-\beta)\ell\}$ and
  $\epsilon>0$. By lemma \ref{lem:aux1} there exists a large number
  $\gamma$ such that $\prob\big[T_{\lfloor\gamma t\rfloor}\le
    t\big]\le e^{-\lambda t}$ for all sufficiently large $t$. Set
  $q_t:=\lfloor\gamma t\rfloor$. Then, (\ref{eq:starting_upper}) gives
  for all sufficiently large $t$
\begin{equation*}
  \prob\bigg[\frac{W_t}{t}\in C\bigg]\le\mathds{1}_{\{0\in C\}}\prob[S_1>t]
  +\sum_{n=1}^{q_t-1}\prob\Bigg[\frac{1}{t}\sum_{i=1}^nX_i\in C,\,T_n\le t<T_{n+1}\Bigg]+e^{-\lambda t}.
\end{equation*}
Since $-\limsup_{t\uparrow+\infty}(1/t)\ln\prob[S_1>t]=:\ells\ge\ell$,
there exists a constant $M>0$ such that $\prob[S_1>t]\le
Me^{(\epsilon-\ell)t}$ for all $t>0$. Moreover, by definition we have
$\lambda<\Upsilon(\beta,w)+(1-\beta)\ell$ for all $\beta\in[0,1]$ and
$w\in C$. Recall that $\Upsilon(0,0)=0$ by lemma
\ref{lem:Upsilon_prop}.  Thus, if $0\in C$, then
$\lambda<\Upsilon(0,0)+\ell=\ell$, which shows that
$\mathds{1}_{\{0\in C\}}\le e^{(\ell-\lambda)t}$ for any $t>0$. It
follows that for all sufficiently large $t$
\begin{equation*}
  \prob\bigg[\frac{W_t}{t}\in C\bigg]\le e^{-\lambda t}+M e^{(\epsilon-\lambda)t}
  +\sum_{n=1}^{q_t-1}\prob\Bigg[\frac{1}{t}\sum_{i=1}^nX_i\in C,\,T_n\le t<T_{n+1}\Bigg].
\end{equation*}
To address the third term in the r.h.s., pick real numbers
$0:=\beta_0<\beta_1<\cdots<\beta_K:=1$ such that
$\beta_k-\beta_{k-1}\le\epsilon$ for each $k$. We have
\begin{align}
  \nonumber
  \sum_{n=1}^{q_t-1}\prob\Bigg[\frac{1}{t}\sum_{i=1}^nX_i\in C,\,T_n\le t<T_{n+1}\Bigg]
&\le\sum_{n=1}^{q_t-1}\sum_{k=1}^K\prob\Bigg[\frac{1}{t}\sum_{i=1}^nX_i\in C,\,\frac{T_n}{t}\in[\beta_{k-1},\beta_k],\,S_{n+1}>t-T_n\Bigg]\\
  \nonumber
  &\le\sum_{n=1}^{q_t-1}\sum_{k=1}^K\prob\Bigg[\frac{1}{t}\sum_{i=1}^nX_i\in C,\,\frac{T_n}{t}\in[\beta_{k-1},\beta_k]\Bigg]\cdot\prob\big[S_1>(1-\beta_k)t\big]\\
\nonumber
&\le M\sum_{n=1}^{q_t-1}\sum_{k=1}^K\prob\Bigg[\frac{1}{t}\sum_{i=1}^n(S_i,X_i)\in[\beta_{k-1},\beta_k]\times C\Bigg]e^{(\epsilon-\ell)(1-\beta_k)t}.
\end{align}
On the other hand, lemma \ref{lem:aux2}, together with the fact that
$\lambda<\Upsilon(s,w)+(1-s)\ell$ for all $s\in[0,1]$ and $w\in C$,
show that for every $t>0$, $n\ge 1$, and $k\le K$
\begin{align}
  \nonumber
  \ln \prob\Bigg[\frac{1}{t}\sum_{i=1}^n(S_i,X_i)\in[\beta_{k-1},\beta_k]\times C\Bigg]
  &\le-t\inf_{(s,w)\in[\beta_{k-1},\beta_k]\times C}\big\{\Upsilon(s,w)\big\}\\
  \nonumber
  &\le-t\inf_{s\in[\beta_{k-1},\beta_k]}\big\{\lambda-(1-s)\ell\big\}\\
  \nonumber
  &\le t\big[(1-\beta_k)\ell+\epsilon|\ell|-\lambda\big].
\end{align}
It follows that for each $t>0$
\begin{equation*}
\sum_{n=1}^{q_t-1}\prob\Bigg[\frac{1}{t}\sum_{i=1}^nX_i\in C,\,T_n\le t<T_{n+1}\Bigg]
\le MKq_t\, e^{(\epsilon+\epsilon|\ell|-\lambda)t}\le \gamma MK t\,e^{(\epsilon+\epsilon|\ell|-\lambda)t}.
\end{equation*}
In conclusion, for all sufficiently large $t$ we get
\begin{equation*}
  \prob\bigg[\frac{W_t}{t}\in C\bigg]\le e^{-\lambda t}+M e^{(\epsilon-\lambda)t}+\gamma MK t\, e^{(\epsilon+\epsilon|\ell|-\lambda)t}, 
\end{equation*}
which shows that 
\begin{equation*}
  \limsup_{t\uparrow+\infty}\frac{1}{t}\ln\prob\bigg[\frac{W_t}{t}\in C\bigg]\le \epsilon+\epsilon|\ell|-\lambda. 
\end{equation*}
The lemma is proved by sending $\lambda$ to $\inf_{w\in
  C}\inf_{\beta\in[0,1]}\{\Upsilon(\beta,w)+(1-\beta)\ell\}$ and
$\epsilon$ to zero.
\end{proof}

We conclude the section by showing that the hypotheses $\ells<+\infty$
and $\Is(0)<+\infty$ in part (d) of theorem \ref{mainth} can not be
relaxed at the same time. Assume that the waiting times satisfy
$\prob[S_1>s]=e^{-s^2}$ for all positive $s$, so that $\ells=+\infty$.
Set $\rew:=\Rl^2$ and for each $i$ consider the reward
$X_i:=(S_i,Z_i)$ with $Z_i$ independent of $S_i$ and distributed
according to the standard Cauchy law: $\prob[Z_i\le
  z]=1/2+(1/\pi)\arctan(z)$ for all $z\in\Rl$. It is a simple exercise
of calculus to show that $\Is(w)=0$ if $w_1=1$ and $\Is(w)=+\infty$ if
$w_1\ne 1$ for all $w:=(w_1,w_2)\in\rew$, so that in particular we
have $\Is(0)=+\infty$. The upper large deviation bound fails for the
open convex set $C:=\{w\in\rew:w_1<1\}$. In fact, $\inf_{w\in
  C}\{\Is(w)\}=+\infty$ and for all $t>0$
\begin{align}
\nonumber
\prob\bigg[\frac{W_t}{t}\in C\bigg]&=\prob\big[S_1>t\big]+\sum_{n\ge 1}\prob\big[T_n<t<T_{n+1}\big]\\
\nonumber
&=\prob\big[T_1>t\big]+\sum_{n\ge 1}\prob\big[T_n\le t<T_{n+1}\big]=1.
\end{align}
It also fails for the closed convex set $C:=\{w\in\rew:w_1<1 \mbox{
  and } (1-w_1)w_2\ge 1\}$. Indeed, $\inf_{w\in C}\{\Is(w)\}=+\infty$
and
\begin{equation}
\limsup_{t\uparrow\infty}\frac{1}{t}\ln \prob\bigg[\frac{W_t}{t}\in C\bigg]=0
\label{ex:closed_convex_1}
\end{equation}
as we now demonstrate. Pick a real number $\epsilon\in(0,1)$ and an
integer $N\ge 1$. Let $\mu$ and $\sigma^2$ be the mean and the
variance of $S_1$, respectively.  Since $(1/N)\sum_{i=1}^NZ_i$ is
distributed as $Z_1$ by the stability property of the Cauchy law, for
every $t>0$ we have
\begin{align}
\nonumber
\prob\bigg[\frac{W_t}{t}\in C\bigg]&\ge\prob\Bigg[T_N<t<T_{N+1},\,\frac{1}{t}\sum_{i=1}^NZ_i\ge \frac{1}{1-T_N/t}\Bigg]\\
\nonumber
&\ge\prob\Bigg[t-\sqrt{\epsilon\sigma^2N}<T_N\le t-\sqrt{\epsilon^2\sigma^2N},\,S_{N+1}>t-T_N,\,\sum_{i=1}^NZ_i\ge \frac{t^2}{t-T_N}\Bigg]\\
\nonumber
&\ge\prob\Big[-\sqrt{\epsilon \sigma^2N}<T_N-t\le -\sqrt{\epsilon^2\sigma^2N}\Big]\cdot\prob\Big[S_1>\sqrt{\epsilon \sigma^2N}\Big]\cdot
\prob\Bigg[\sum_{i=1}^NZ_i\ge \frac{t^2}{\sqrt{\epsilon^2\sigma^2N}}\Bigg]\\
\nonumber
&=\prob\Big[-\sqrt{\epsilon \sigma^2 N}<T_N-t\le -\sqrt{\epsilon^2\sigma^2N}\Big]\cdot e^{-\epsilon \sigma^2N}\cdot
\prob\bigg[\epsilon\sigma Z_1\ge\frac{t^2}{\sqrt{N^3}}\bigg].
\end{align}
At this point, by taking $t=\mu N$ and by sending $N$ to infinity we
obtain
\begin{equation}
\limsup_{t\uparrow\infty}\frac{1}{t}\ln \prob\bigg[\frac{W_t}{t}\in C\bigg]\ge-\frac{\epsilon \sigma^2}{\mu}
\label{ex:closed_convex_2}
\end{equation}
because $\lim_{N\uparrow\infty}\prob[-\sqrt{\epsilon
    \sigma^2N}<T_N-\mu N\le
  -\sqrt{\epsilon^2\sigma^2N}]=(1/\sqrt{2\pi})\int_{-\sqrt{\epsilon}}^{-\epsilon}
e^{-\frac{1}{2}\zeta^2}d\zeta>0$ by the central limit theorem and
$\lim_{N\uparrow\infty}\mu^2\sqrt{N}\,\prob[\epsilon\sigma
  Z_1\ge\mu^2\sqrt{N}]=\epsilon\sigma/\pi$. The arbitrariness of
$\epsilon$ in (\ref{ex:closed_convex_2}) gives (\ref{ex:closed_convex_1}).

\subsection{The upper large deviation bound for compact sets}
\label{sec:compactsets}

The upper large deviation bound for compact sets is basically due to
the following lemma.
\begin{lemma}
\label{lem:compact_aux}
For each $w\in\rew$ and real number $\lambda<\Is(w)$ there exists
$\delta>0$ such that
\begin{equation*}
  \limsup_{t\uparrow+\infty}\frac{1}{t}\ln\prob\bigg[\frac{W_t}{t}\in B_{w,\delta}\bigg]\le -\lambda.
\end{equation*}
\end{lemma}

\begin{proof}
Pick $w\in\rew$ and a real number $\lambda<\Is(w)$. If
$\ells<+\infty$, then the lemma is an immediate application of part
(d) of theorem \ref{mainth}. In fact, by the lower semicontinuity of
$\Is$ there exists $\delta>0$ such that $\Is(v)\ge\lambda$ for all
$v\in B_{w,\delta}$, and the open ball $B_{w,\delta}$ is
convex. Assume $\ells=+\infty$. In such case we have
$\lambda<\Is(w)=\Upsilon(1,w)$ and by the lower semicontinuity of
$\Upsilon$ there exists $\delta\in(0,1)$ such that $\Upsilon(s,v)\ge\lambda$
for every $s\in[1-\delta,1]$ and $v\in B_{w,\delta}$. It follows by lemma 
\ref{lem:aux2} that for all $n\ge 1$ and $t>0$
\begin{equation*}
  \ln \prob\Bigg[\frac{1}{t}\sum_{i=1}^n(S_i,X_i)\in[1-\delta,1]\times B_{w,\delta}\Bigg]
  \le-t\inf_{(s,w)\in[1-\delta,1]\times B_{w,\delta}}\big\{\Upsilon(s,v)\big\}\le -\lambda t.
\end{equation*}
Moreover, lemma \ref{lem:aux1} gives that $\prob[T_{\lfloor\gamma
    t\rfloor}\le t]\le e^{-\lambda t}$ for all sufficiently large $t$
and some real number $\gamma>0$. Setting $q_t:=\lfloor\gamma
t\rfloor$, bound (\ref{eq:starting_upper}) yields for all sufficiently
large $t$
\begin{align}
  \nonumber
  \prob\bigg[\frac{W_t}{t}\in B_{w,\delta}\bigg]&\le \prob[S_1>t]
  +\sum_{n=1}^{q_t-1}\prob\Bigg[\frac{1}{t}\sum_{i=1}^nX_i\in B_{w,\delta},\,T_n\le t<T_{n+1}\Bigg]+e^{-\lambda t}\\
  \nonumber
  &\le q_t\prob[S_1>\delta t]
  +\sum_{n=1}^{q_t-1}\prob\Bigg[\frac{1}{t}\sum_{i=1}^nX_i\in B_{w,\delta},\,t-S_{n+1}<T_n\le t,\,S_{n+1}\le\delta t\Bigg]+e^{-\lambda t}\\
  \nonumber
  &\le q_t\prob[S_1>\delta t]
  +\sum_{n=1}^{q_t-1}\prob\Bigg[\frac{1}{t}\sum_{i=1}^nX_i\in B_{w,\delta},\,\frac{T_n}{t}\in[1-\delta,1]\Bigg]+e^{-\lambda t}\\
  \nonumber
  &\le \gamma t\,\prob[S_1>\delta t]+\gamma t e^{-\lambda t}
\end{align}
This inequality proves the lemma since
$\limsup_{t\uparrow+\infty}(1/t)\ln\prob[S_1>\delta t]=-\infty$ as
$\ells=+\infty$.
\end{proof}

Let us verify part (c) of theorem \ref{mainth}. Pick a compact set $F$
in $\rew$ and a real number $\lambda<\inf_{w\in F}\{\Is(w)\}$. Lemma
\ref{lem:compact_aux} guarantees that for each $w\in F$ there exists
$\delta_w>0$ such that
\begin{equation*}
  \limsup_{t\uparrow+\infty}\frac{1}{t}\ln\prob\bigg[\frac{W_t}{t}\in B_{w,\delta_w}\bigg]\le -\lambda.
\end{equation*}
Since $F$ is compact, we can find a finite number of points
$w_1,\ldots,w_K$ in $F$ such that $F\subset\cup_{k=1}^K
B_{w_k,\delta_{w_k}}$. It follows that
\begin{equation*}
  \limsup_{t\uparrow+\infty}\frac{1}{t}\ln\prob\bigg[\frac{W_t}{t}\in F\bigg]\le
  \limsup_{t\uparrow+\infty}\frac{1}{t}\ln\sum_{k=1}^K\prob\bigg[\frac{W_t}{t}\in B_{w_k,\delta_{w_k}}\bigg]\le -\lambda,
\end{equation*}
which yields
\begin{equation*}
  \limsup_{t\uparrow+\infty}\frac{1}{t}\ln\prob\bigg[\frac{W_t}{t}\in F\bigg]\le-\inf_{w\in F}\big\{\Is(w)\big\}
\end{equation*}
once $\lambda$ is sent to $\inf_{w\in F}\{\Is(w)\}$.

\subsection{The upper large deviation bound for closed sets}
\label{sec:closedsets}

The upper large deviation bound can be extended from compact sets to
close sets if the probability distribution of $W_t/t$ is exponential
tight (see \cite{DemboBook}, lemma 1.2.18), namely if for each real
number $\lambda>0$ there exists a compact set $K$ in $\rew$ such that
\begin{equation*}
  \limsup_{t\uparrow+\infty}\frac{1}{t}\ln\prob\bigg[\frac{W_t}{t}\notin K\bigg]\le-\lambda.
\end{equation*}
The following lemma establishes exponential tightness of the scaled
cumulant reward when $\rew$ has finite dimension and proves part (e)
of theorem \ref{mainth}.
\begin{lemma}
\label{closed_constrained}
Assume that $\rew$ has finite dimension and that there exist numbers
$\zeta\le 0$ and $\sigma>0$ such that $\Ex[e^{\zeta
    S_1+\sigma\|X_1\|}]<+\infty$. Then, the probability distribution
of $W_t/t$ is exponential tight. Moreover, $\Ii$ and $\Is$ have
compact level sets.
\end{lemma}

\begin{proof}
To begin with, let us observe that under the hypotheses of the lemma
we have $\lim_{\eta\downarrow-\infty}\Ex[e^{\eta
    S_1+\sigma\|X_1\|}]=0$ by the dominated convergence theorem, so
that there exists $\eta\le 0$ that satisfies $\Ex[e^{\eta
    S_1+\sigma\|X_1\|}]\le 1/2$.  Let $d$ be the dimension of $\rew$,
let $\{v_1,\ldots,v_d\}$ be a basis of $\rew$, and let
$\{\vartheta_1,\ldots,\vartheta_d\}\subset\rew^\star$ be the dual
basis: $\vartheta_k(v_l)$ equals 1 if $k=l$ and 0 otherwise for all
$k$ and $l$. For $k$ ranging from $1$ to $d$ set
$\varphi_k:=\vartheta_k/\|\vartheta_k\|$ and
$\varphi_{d+k}:=-\varphi_k$. We have $\Ex[e^{\eta
    S_1+\sigma\varphi_k(X_1)}]\le\Ex[e^{\eta S_1+\sigma\|X_1\|}]\le
1/2$ for every $k$. Fix a real number $\lambda>0$ and introduce the
compact set $K:=\cap_{k=1}^{2d}\{w\in\rew:\varphi_k(w)\le \rho\}$,
where we have set $\rho:=(\lambda-\eta)/\sigma>0$ for brevity. Since
$0$ does not belong to the complement $K^c$ of $K$, for all $t>0$
\begin{equation*}
  \prob\bigg[\frac{W_t}{t}\notin K\bigg]=\sum_{n\ge 1}\prob\Bigg[\frac{1}{t}\sum_{i=1}^nX_i\in K^c,\,T_n\le t<T_{n+1}\Bigg].
\end{equation*}
As $K^c=\cup_{k=1}^{2d}\{w\in\rew:\varphi_k(w)> \rho\}$, by making use
of the Chernoff bound twice and by recalling that $\Ex[e^{\eta
    S_1+\sigma\varphi_k(X_1)}]\le 1/2$ for any $k$, we obtain for
every $t$
\begin{align}
    \nonumber
  \prob\bigg[\frac{W_t}{t}\notin K\bigg]&\le\sum_{n\ge 1}\sum_{k=1}^{2d}\prob\Bigg[\sum_{i=1}^n\varphi_k(X_i)>\rho t,\,T_n\le t\Bigg]\\
  \nonumber
  &\le e^{-\eta t-\sigma\rho t}\sum_{k=1}^{2d}\sum_{n\ge 1}\Ex\Big[e^{\eta S_1+\sigma\varphi_k(X_1)}\Big]^n\le 2d\,e^{-\lambda t}.
\end{align}
This inequality proves exponential tightness of the distribution of $W_t/t$.

To conclude, let us show that the level sets of $\Ii$ and $\Is$ are
compact.  Regarding $\Ii$, compactness of level sets follows by
combining the lower large deviation bound of part (b) of theorem
\ref{mainth} and the exponential tightness (see \cite{DemboBook},
lemma 1.2.18). Let us move to $\Is$.  To begin with, we observe that
for all $s\in\Rl$, $w\in\rew$, and $k\le 2d$
\begin{equation*}
J(s,w)\ge s\eta+\sigma\varphi_k(w)-\ln\Ex\big[e^{\eta S_1+\sigma\varphi_k(X_1)}\big]\ge s\eta+\sigma\varphi_k(w)
\end{equation*}
by definition. It follow that $\Upsilon(\beta,w)\ge
\beta\eta+\sigma\varphi_k(w)$ for each $\beta\in[0,1]$, $w\in\rew$,
and $k\le 2d$, which yields $\Is(w)\ge \eta+\sigma\varphi_k(w)$ for
every $w$ and $k$.  Thus, if $w$ is such that $\Is(w)\le\lambda$ for a
given real number $\lambda\ge 0$, then $\varphi_k(w)\le
(\lambda-\eta)/\sigma$ for all $k$. This demonstrates that level set
$\{w\in\rew:\Is(w)\le\lambda\}$ is bounded. It is closed by the lower
semicontinuity of $\Is$.
\end{proof}

The case in which $\rew$ has infinite dimension is solved by the next
lemma.

\begin{lemma}
Assume that $\rew$ has infinite dimension and that $\Ex[e^{\sigma
    S_1+\sigma\|X_1\|}]<+\infty$ for all $\sigma>0$. Then, the
probability distribution of $W_t/t$ is exponential tight. Moreover,
$\ells=+\infty$ and $\Upsilon(1,\cdot\,)$ has compact level sets.
\end{lemma}

\begin{proof}
  Fix $\lambda>0$ and set $\epsilon:=\lambda/\ln\Ex[e^{2\lambda
      S_1}]$, which is positive since $\Ex[e^{2\lambda S_1}]<+\infty$
  by hypothesis. As $\Ex[e^{\sigma\|X_1\|}]<+\infty$ for all
  $\sigma>0$, the theory of Cram\'er in separable Banach spaces tells
  us that the distribution of $(1/n)\sum_{i=1}^nX_i$ is exponential
  tight (see \cite{DemboBook}, exercise 6.2.21). Thus, there exists a
  compact set $K_o$ in $\rew$ such that for all sufficiently large $n$
\begin{equation}
\prob\Bigg[\frac{1}{n}\sum_{i=1}^nX_i\notin K_o\Bigg]\le e^{-\lambda n/\epsilon}.
\label{Cramer_tight}
\end{equation}
Let $C$ be the closure of the convex hull of $\{0\}\cup K_o$, which is
compact (see \cite{Rudin}, theorem 3.20). Let $\gamma>\epsilon$ be a
real number such that $\prob[T_{\lfloor \gamma t\rfloor}\le t]\le
e^{-\lambda t}$ for all $t$ large enough, which exists by lemma
\ref{lem:aux1}, and set $K:=\{w\in\rew:\gamma w\in C\}$, which clearly
is a compact set. Set $p_t:=\lfloor \epsilon t\rfloor$ and
$q_t:=\lfloor\gamma t\rfloor$.  For all sufficiently large $t$ we have
$1\le p_t<q_t$ and $\prob[T_{p_t}>t]\le\Ex[e^{2\lambda
    S_1}]^{p_t}e^{-2\lambda t}\le\Ex[e^{2\lambda S_1}]^{\epsilon
  t}e^{-2\lambda t}=e^{-\lambda t}$ by the Chernoff bound. Then, for
all sufficiently large $t$ we can write
\begin{align}
  \nonumber
  \prob\bigg[\frac{W_t}{t}\notin K\bigg]&\le\prob\big[T_{p_t}>t\big]
  +\prob\bigg[\frac{W_t}{t}\notin K,\,T_{p_t}\le t<T_{q_t}\bigg]+\prob\big[T_{q_t}\le t\big]\\
\nonumber
  &\le 2e^{-\lambda t}+\sum_{n=p_t}^{q_t-1}\prob\Bigg[\frac{1}{t}\sum_{i=1}^nX_i\notin K\Bigg].
\end{align}
On the other hand, the condition $(1/n)\sum_{i=1}^nX_i\in C$ implies
$(n/\gamma t)(1/n)\sum_{i=1}^nX_i\in C$, namely
$(1/t)\sum_{i=1}^nX_i\in K$, for $n<q_t$ as $C$ is convex and contains
the origin. This shows that if $(1/t)\sum_{i=1}^nX_i\notin K$ for
$n<q_t$, then $(1/n)\sum_{i=1}^nX_i\notin C$. Thus, by recalling that
$K_0\subseteq C$ and by invoking (\ref{Cramer_tight}), for all
sufficiently large $t$ we find
\begin{align}
  \nonumber
  \prob\bigg[\frac{W_t}{t}\notin K\bigg]&\le 2e^{-\lambda t}+\sum_{n=p_t}^{q_t-1}\prob\Bigg[\frac{1}{n}\sum_{i=1}^nX_i\notin C\Bigg]\\
  \nonumber
  &\le 2e^{-\lambda t}+\sum_{n=p_t}^{q_t-1}\prob\Bigg[\frac{1}{n}\sum_{i=1}^nX_i\notin K_o\Bigg]
  \le 2e^{-\lambda t}+\gamma t\,e^{-\lambda p_t/\epsilon}.
\end{align}
This bound yields
\begin{equation*}
 \limsup_{t\uparrow+\infty}\frac{1}{t}\ln\prob\bigg[\frac{W_t}{t}\notin K\bigg]\le -\lambda.
\end{equation*}

As before, the lower large deviation bound of part (b) of theorem
\ref{mainth} and the exponential tightness imply that $\Ii$ has
compact level sets. On the other hand, since $e^{\sigma
  s}\,\prob[S_1>s]\le \Ex[e^{\sigma S_1}]<+\infty$ for all $\sigma>0$
and $s>0$, we have $\ells=\elli=+\infty$. Thus,
$\Upsilon(1,\cdot\,)=\Ii$.
\end{proof}


\appendix

\section{Proof of proposition \ref{prop:ratewait}}
\label{proof:ratewait}

By definition we have $\Is\le \Ii\le \Upsilon(1,\cdot\,)$. Let us show
that $\Is\ge\Upsilon(1,\cdot\,)$ under the hypotheses of the
proposition, which is nontrivial only when $\ells<+\infty$. Suppose
that $\ells<+\infty$. In this case we have $\Ex[e^{\zeta
    S_1+\varphi(X_1)}]=+\infty$ for all $\zeta>\ells$ and
$\varphi\in\rew^\star$. In fact, given $\zeta>\ells$ and
$\varphi\in\rew^\star$ one can find a real number $\epsilon>0$ such
that $\zeta-\epsilon>\ells\ge 0$ and $\|\varphi\|f(s)\le\epsilon s$
for all sufficiently large $s$.  Then, $\Ex[e^{\zeta
    S_1+\varphi(X_1)}]\ge e^{(\zeta-\epsilon)t}\,\prob[S_1>t]$ for all
sufficiently large $t$. It follows that $\Ex[e^{\zeta
    S_1+\varphi(X_1)}]=+\infty$ since $\zeta-\epsilon>\ells$. This
way, for all $\gamma>0$, $s\le 1$, and $w\in\rew$ we find
\begin{align}
\nonumber
\gamma J(s/\gamma,w/\gamma)&=\sup_{(\zeta,\varphi)\in(-\infty,\ells]\times\rew^\star}\Big\{s\zeta+\varphi(w)-\gamma\ln\Ex\big[e^{\zeta S_1+\varphi(X_1)}\big]\Big\}\\
\nonumber
&\ge\sup_{(\zeta,\varphi)\in(-\infty,\ells]\times\rew^\star}\Big\{\zeta+\varphi(w)-\gamma\ln\Ex\big[e^{\zeta S_1+\varphi(X_1)}\big]\Big\}+(s-1)\ells\\
  \nonumber
  &=\gamma J(1/\gamma,w/\gamma)+(s-1)\ells.
\end{align}
The definition of $\Upsilon$ and part $(iv)$ of lemma
\ref{lem:Upsilon_prop} then yields $\Upsilon(\beta,w)\ge
\Upsilon(1,w)+(\beta-1)\ells$ for every $\beta\in[0,1]$ and
$w\in\rew$, so that $\Is\ge \Upsilon(1,\cdot\,)$.

\section{Cram\'er's theory for waiting times and rewards}
\label{Cramer}

This appendix introduces the basics of Cram\'er's theory that are used
to prove theorem \ref{mainth}. Let the space $\Rl\times\rew$ be
endowed with the product topology and the Borel $\sigma$-field
$\BWrew$ and consider the measure $\mu_n$ over $\BWrew$ defined for
each integer $n\ge 1$ by
\begin{equation*}
\mu_n:=\prob\Bigg[\frac{1}{n}\sum_{i=1}^n (S_i,X_i)\in \cdot\,\Bigg].
\end{equation*}
Of fundamental importance is the following super-multiplicativity
property.
\begin{lemma}
  Let $C\in\BWrew$ be convex and let $m\ge 1$ and $n\ge 1$ be two
  integers. Then, $\mu_{m+n}(C)\ge\mu_m(C)\cdot\mu_n(C)$.
\end{lemma}

\begin{proof}
See lemma 6.1.12 of \cite{DemboBook}.
\end{proof}

Super-multiplicativity, which becomes super-additivity once logarithms
are taken, makes it possible to describe in general terms the
exponential decay with $n$ of the measure $\mu_n$. To this purpose, we
denote by $\mathcal{L}$ the extended real function over $\BWrew$
defined by the formula
\begin{equation*}
\mathcal{L}:=\sup_{n\ge 1}\bigg\{\frac{1}{n}\ln \mu_n\bigg\}.
\end{equation*}
If $C\in\BWrew$ is convex, then $\limsup_{n\uparrow\infty}(1/n)\ln
\mu_n(C)=\mathcal{L}(C)$ due the super-additivity of $\ln
\mu_n(C)$. The following lemma improves this result when $C$ is open
as well as convex.
\begin{lemma}
  \label{start}
  Let $C\subseteq\Rl\times\rew$ be open and convex. Then,
  $\lim_{n\uparrow\infty}(1/n)\ln \mu_n(C)$ exists as an extended real
  number and is equal to $\mathcal{L}(C)$.
\end{lemma}

\begin{proof}
See lemma 1.1.14 of \cite{DemboBook}.
\end{proof}

Lemma \ref{start} prompts one to consider the rate function $J$
that maps any $(s,w)\in\Rl\times\rew$ in the extended real number
$J(w)$ defined by
\begin{equation*}
J(s,w):=-\inf\Big\{\mathcal{L}(C): C\subseteq\Rl\times\rew \mbox{ is open convex and contains }(s,w)\Big\}.
\end{equation*}
In fact, the following weak large deviation principle is verified.
\begin{proposition}
\label{WLDP}
The following conclusions hold:
\begin{enumerate}[(i)]
\item the function $J$ is lower semicontinuous and convex;
\item $\displaystyle{\liminf_{n\uparrow\infty}\frac{1}{n}\ln\mu_n(G)\ge -\inf_{(s,w)\in G}\{J(s,w)\}}$ for each $G\subseteq\Rl\times\rew$ open;
\item $\displaystyle{\limsup_{n\uparrow\infty}\frac{1}{n}\ln\mu_n(K)\le -\inf_{(s,w)\in K}\{J(s,w)\}}$ for each $K\subseteq\Rl\times\rew$ compact.
\end{enumerate} 
\end{proposition}

\begin{proof}
See lemma 6.1.7 of \cite{DemboBook}.
\end{proof}

The rate function $J$ can be related to the moment generating function
of waiting time and reward pairs as follows.
\begin{proposition}
  \label{Jesplicita}
  For all $(s,w)\in\Rl\times\rew$
  \begin{equation*}
   J(s,w)=\sup_{(\zeta,\varphi)\in\Rl\times\rew^\star}\Big\{s\zeta+\varphi(w)-\ln\Ex\big[e^{\zeta S_1+\varphi(X_1)}\big]\Big\}.
  \end{equation*}
\end{proposition}

\begin{proof}
See theorem 6.1.3 of \cite{DemboBook}.
\end{proof}

We conclude the appendix with a result about certain convex sets that are
met in the proof of theorem \ref{mainth}.
\begin{lemma}
\label{lemma_conv}
Let $\alpha<\beta$ be two real numbers and let $C\subseteq\rew$ be
open convex, closed convex, or any convex set in $\Brew$ when $\rew$
is finite-dimensional.  Then, for all $n\ge 1$
\begin{equation*}
  \frac{1}{n}\ln \mu_n\big([\alpha,\beta]\times C\big)
  \le-\inf_{(s,w)\in[\alpha,\beta]\times C}\big\{J(s,w)\big\}.
\end{equation*}
\end{lemma}

\begin{proof}
Recalling the definition of $\mathcal{L}$, we show that 
\begin{equation*}
\mathcal{L}([\alpha,\beta]\times C)\le-\inf_{(s,w)\in[\alpha,\beta]
\times C}\{J(s,w)\}.
\end{equation*} 
Assume $\mathcal{L}([\alpha,\beta]\times C)>-\infty$, otherwise there
is nothing to prove, and pick $\epsilon>0$. By definition, there
exists an integer $N\ge 1$ such that $\mathcal{L}([\alpha,\beta]\times
C)\le(1/N)\ln\mu_N([\alpha,\beta]\times C)+\epsilon$.  Notice that we
must have $\mu_N([\alpha,\beta]\times C)>0$.  Completeness and
separability of $\rew$ entail that the measure that associates any
$A\in\Brew$ with $\mu_N([\alpha,\beta]\times A)$ is tight (see
\cite{Bogachevbook}, theorem 7.1.7).  Consequently, a compact set
$K_o\subseteq C$ can be found so that $\mu_N([\alpha,\beta]\times
C)\le \mu_N([\alpha,\beta]\times K_o)+[1-\exp(-\epsilon
  N)]\mu_N([\alpha,\beta]\times C)$.  Thus,
$\mu_N([\alpha,\beta]\times C)\le\exp(\epsilon
N)\mu_N([\alpha,\beta]\times K_o)$ and
$\mathcal{L}([\alpha,\beta]\times
C)\le(1/N)\ln\mu_N([\alpha,\beta]\times K_o)+2\epsilon$ follows.  We
shall show in a moment that there exists a compact convex set $K$ with
the property that $K_o\subseteq K\subseteq C$.  Then, using the fact
that $K_o\subseteq K$ we reach the further bound
$\mathcal{L}([\alpha,\beta]\times
C)\le(1/N)\ln\mu_N([\alpha,\beta]\times
K)+2\epsilon\le\mathcal{L}([\alpha,\beta]\times K)+2\epsilon$.  At
this point, we observe that on the one hand
$\mathcal{L}([\alpha,\beta]\times
K)=\limsup_{n\uparrow\infty}(1/n)\ln\mu_n([\alpha,\beta]\times K)$ by
super-additivity as $K$ is convex, and on the other hand
$\limsup_{n\uparrow\infty}(1/n)\ln\mu_n([\alpha,\beta]\times
K)\le-\inf_{(s,w)\in[\alpha,\beta]\times K}\{J(s,w)\}$ by part $(ii)$
of proposition \ref{WLDP} as $K$ is compact. Thus,
$\mathcal{L}([\alpha,\beta]\times
C)\le-\inf_{(s,w)\in[\alpha,\beta]\times
  K}\{J(s,w)\}+2\epsilon\le-\inf_{(s,w)\in[\alpha,\beta]\times
  C}\{J(s,w)\}+2\epsilon$ because $K\subseteq C$ and the lemma follows
from the arbitrariness of $\epsilon$.

Let us prove at last that there exists a compact convex set $K$ with
the property that $K_o\subseteq K\subseteq C$.  The hypothesis that
the convex set $C$ is either open or closed when $\rew$ if
infinite-dimensional comes into play here.  Let $C_o$ be the convex
hull of the compact set $K_o\subseteq C$ and denote the closure of a
set $A$ by $\cl A$. The set $C_o$ is convex and compact when $\rew$ is
finite-dimensional, whereas $\cl C_o$ is convex and compact even when
$\rew$ is infinite-dimensional (see \cite{Rudin}, theorem
3.20). Clearly, $K_o\subseteq C_o\subseteq C$.  If $\rew$ is
finite-dimensional, then the problem to find $K$ is solved by
$K=C_o$. If $\rew$ is infinite-dimensional and $C$ is closed, then the
problem is solved by $K=\cl C_o$. Some more effort is needed when
$\rew$ is infinite-dimensional and $C$ is open. Assume that $C$ is
open and for each $w\in C$ let $\delta_w>0$ be such that $\cl
B_{w,\delta_w}\subseteq C$. As $K_o$ is compact, there exist finitely
many points $w_1,\ldots,w_n$ in $K_o$ so that
$K_o\subseteq\cup_{i=1}^nB_{w_i,\delta_{w_i}}$. Let $K$ be the convex
hull of $\cup_{i=1}^n(\cl B_{w_i,\delta_{w_i}}\cap \cl C_o)$, which
contains $K_o$. The set $K$ is convex and compact since it is the
convex hull of the union of the compact convex sets $\cl
B_{w_1,\delta_{w_1}}\cap \cl C_o,\ldots,\cl B_{w_n,\delta_{w_n}}\cap
\cl C_o$ (see \cite{Rudin}, theorem 3.20). On the other hand, we have
$K\subseteq C$ because $\cup_{i=1}^n(\cl B_{w_i,\delta_{w_i}}\cap \cl
C_o)\subseteq\cup_{i=1}^n\cl B_{w_i,\delta_{w_i}}\subseteq C$.
\end{proof}



\end{document}